\documentclass[11pt]{amsart}

\usepackage[spanish,USenglish]{babel} 
\usepackage[utf8]{inputenc}
\usepackage{float}
\usepackage{amssymb}
\usepackage{amsmath}
\usepackage{amsthm}
\usepackage{subfig}
\usepackage{hyperref}
\usepackage{pinlabel}
\usepackage{color}

\makeatletter
\renewcommand*\env@matrix[1][*\c@MaxMatrixCols c]{%
  \hskip -\arraycolsep
  \let\@ifnextchar\new@ifnextchar
  \array{#1}}
\makeatother

\def\ord{\operatorname{ord}}
\def\Sym{\operatorname{Sym}}

\def\Homeo{\operatorname{Homeo}}

\newtheorem{thm}{Theorem}[section]
\newtheorem{prop}[thm]{Proposition}
\newtheorem{corollary}[thm]{Corollary}
\newtheorem{lemma}[thm]{Lemma}
\theoremstyle{definition}
\newtheorem{remark}[thm]{Remark}
\newtheorem{defi}[thm]{Definition}

\begin{document}
\title[The conjugacy problem for Thompson-like groups]{The conjugacy problem for \\ Thompson-like groups}	
\author{Julio Aroca}
\address{Instituto de Ciencias Matemáticas, C/ Nicolás Cabrera, 13-15, E-28049 Madrid, Spain}
\email{julio.aroca@icmat.es}
\date{\today}
\begin{abstract}
In this paper we generalize techniques of Belk-Matucci \cite{BM} to solve the conjugacy problem for every Thompson-like group $V_n(H)$, where $n \geq 2$ and $H$ is a subgroup of the symmetric group on $n$ elements. We use this to prove that, if $n \neq m$, $V_n(H)$ is not isomorphic to $V_m(G)$ for any $H,G$.
\end{abstract}
\maketitle

\section{Introduction}

In \cite{HI}, Higman constructed a family of finitely-presented infinite simple groups $V_n$, of which $V_2$ is the celebrated group $V$ previously defined by Thompson in the 1960's. Later, these groups were extended in \cite{FH, NE} to the family of Thompson-like groups $V_n(H)$, where $H$ is any subgroup of the symmetric group on $n$ elements, $\Sym(n)$. Using this notation, $V_n(Id) = V_n$ of Higman. We remark that the groups $V_n(Id)$ and $V_n(\mathbb{Z}_2)$ also appear in the study of groups of homeomorphisms of surfaces of infinite topological type \cite{AF2, FN}.

The groups $V_n(H)$ may be informally described as follows; see Section \ref{S2} for a detailed  definition. Denote by $T_n$ the infinite $n$-regular rooted tree. Observe that the space of ends of $T_n$, namely the space of infinite sequences on the alphabet $\{1, \ldots, n\}$, is homeomorphic to the $n$-adic Cantor set $C_n$. Let $\mathcal{T}$ and $\mathcal{T}'$ be rooted subtrees of $T_n$ with the same number $k$ of leaves, both rooted at the root of $T_n$. Let $\tau \in \Sym(k)$ be a bijection between the sets of leaves of $\mathcal{T}$ and $\mathcal{T}'$, respectively. Finally, let $\sigma = (\sigma_1, \dots, \sigma_{k})$ be a tuple of elements $\sigma_i \in H$. Then $(\mathcal{T},\mathcal{T}',\tau,\sigma)$ induces a homeomorphism of $C_n$ by taking the $i$-th leaf of $\mathcal{T}$ to the $\tau(i)$-th leaf in $\mathcal{T}'$ acting on every word below according to $\sigma_{\tau(i)}$. Of course, there are many tuples that induce the same map of $C_n$, and $V_n(H)$ is defined as a group of equivalence classes of these tuples. A well-known result of Brouwer implies that $C_n$ is homeomorphic to the standard Cantor set $C$, so we may regard $V_n(H)$ as a subgroup of $\Homeo(C)$ for all $n\geq 2, H < \Sym(n)$. 

In general, it is difficult to decide whether two arbitrary groups $V_n(G)$ and $V_m(H)$ are isomorphic to each other. Some results have been obtained in this direction \cite{BDJ,FH}, but a complete classification remains unknown. A useful approach to this problem is a result of Higman \cite{HI}, which studies the conjugacy classes (in $V_n(Id)$) of homomorphisms of finite cyclic groups into $V_n(Id)$. This allows him to prove that $V_n(Id) \not\simeq V_m(Id)$ if and only if $n \neq m$. We will be able to reproduce a similar result by studying the conjugacy classes of finite-order elements of some Thompson-like groups. To do that, we base our work in the theory of \textit{strand diagrams} \cite{BM}, a family of directed graphs which are used to solve the conjugacy problem for $V$ and its subgroups $F$ and $T$. We generalize strand diagrams to \textit{$n$-strand diagrams}, a bigger family of graphs which can represent elements of any Thompson-like group, thus the main result of this paper solves the conjugacy problem for every $V_n(H)$. In general terms, two elements of $V_n(H)$ are conjugate if and only if they are represented by the same $n$-strand diagram up to \textit{conjugating transformations}, see Section 3.

\begin{thm}
\label{MT}
Let $n\geq 2$, and $H  \leq \Sym(n)$. Let $f$ and $g$ be two elements of $V_n(H)$. Then $f$ and $g$ are conjugate if and only if their corresponding reduced closed $n$-strand diagrams are equal or they differ by a finite number of conjugating transformations. 
\end{thm}

Once this is done, we prove the following non-isomorphism theorem, using the ideas described previously: 

\begin{thm}
\label{thm:nonisomorphism}
Let $n,m \geq 2$ two different numbers. Then $V_n(P)$ and $V_m(Q)$ are not isomorphic for any $P \leq \Sym(n)$ and $Q \leq \Sym(m)$.
\end{thm}

The plan of the paper is as follows. In Section 2 we will define the Thompson-like groups. Section 3 is devoted to the definition of $n$-strand diagrams and the transformations which can be performed on them. Finally, in Sections 4 and 5 we prove Theorems \ref{MT} and \ref{thm:nonisomorphism}, respectively.


\section{Self-similar bijections and Thompson-like groups}\label{S2}

In this section we will define the Thompson-like groups \cite{FH, NE}. A well known subfamily of these groups are Higman-Thompson's groups, widely covered in \cite{CF, HI}. We first need some definitions.

Let $C_n$ be the $n$-adic Cantor set, which is constructed inductively as follows: $C_n^1$ corresponds to first subdividing $C_n^0= [0,1]$ into $n+1$ intervals of equal length, numbered $1, \ldots, n$ from left to right, and then taking collection of odd-numbered subintervals. Next, $C_n^2$ is obtained from $C_n^1$ by applying the same procedure to each of the intervals forming $C_n^1$, and so on. Finally, $C_n$ is the intersection of all $C_n^i$. Observe that $C_2$ is the usual ``middle-third" Cantor set. 

\subsection{Trees} 
Let $\mathcal{T}_n$ be the regular $n$-ary rooted tree, that is, the infinite simplicial rooted tree where the root has exactly $n$ neighbours, and every vertex other than the root has exactly $n+1$ neighbours. Then, $\mathcal{T}_n$ becomes a metric space by deeming every edge of $\mathcal{T}_n$ to have length 1. For $i\ge 0$, define the $i$-th generation $\mathcal{T}^i_n$ of $\mathcal{T}_n$ to be the set of vertices of $\mathcal{T}_n$ at distance exactly $i$ from the root. It is straightforward to verify that this identification induces a homeomorphism between $C_n$ and the set of ends of $\mathcal{T}_n$. Besides, trees allow us to give a useful alternate incarnation of $C_n$. For $i\ge 1$, label the leaves of $C_n^i$ with the set $\{1,\dots, n\}^{i}$, in such way that labels increase lexicographically from left to right. 
Then, every geodesic ray in $\mathcal{T}_n$ issued from the root is uniquely encoded by an infinite word in $\{1, \dots, n\}$. In this way, we deduce that 
$C_n$ is also homeomorphic to the space $\{1, \dots, n\}^\mathbb{N}$ equipped with the product topology. 

\subsection{Branches} We now introduce some useful subsets of the Cantor space $C_n$ which we will need to define Thompson-like groups. 

\begin{defi}[Concatenation]
Let $u,v$ be words in $\{1, \ldots, n\}$, where $u$ is finite. The \textit{concatenation} $uv$ is the word obtained by joining $v$ after $u$.
\end{defi}

\begin{defi}[Branch]
Let $u$ be a finite word in $\{1, \ldots, n\}$. The {\em branch} of $C_{n}$  determined by $u$ is the subset of $C_{n}$ defined as
\begin{equation*}\label{identification2}
B_{u}=\{ uw \in C_{n} \mid w\in \{1, \ldots, n\}^{\mathbb{N}} \},
\end{equation*} 
where  $uw$ denotes the concatenation of $u$ and $w$. The word $u$ is called the {\em prefix} of the branch $B_u$. 
A {\em  branch} of $C_{n}$ is a subset of the form $B_{u}$, for some finite word $u \in \{1, \ldots, n\}$.
\end{defi}

Note that there is a natural identification of $B_{u}$ with the set of geodesic rays in $\mathcal{T}_n$ issued from $u$, and also with the geodesic in $\mathcal{T}_n$ from the root to $u$, see Figure \ref{fig:branch}.  
In particular, $B_{u}$ is again homeomorphic to the space of ends of the $n$-ary rooted tree, where this time the root is labelled $u$. As such, every branch of $C_{n}$ is itself homeomorphic to $C_n$. 

\begin{figure}[h]
\labellist
\pinlabel $1$ at 24 84
\pinlabel $3$ at 30 63
\pinlabel $u=13$ at 50 47
\pinlabel $B_u$ at 60 25
\endlabellist
\centering
\includegraphics[width=0.25\textwidth]{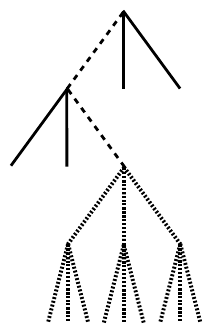}
\caption{An example of branch $B_u \subset C_3$ determined by $u = 13$.}
\label{fig:branch}
\end{figure}
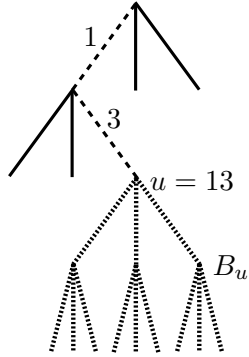

Let $u$, $u'$ be two finite words in $\{1, \ldots, n\}$. We write $u' \subset u$ to mean that there exists a nonempty finite word $x \in \{1, \ldots, n\}$ such that $u' = ux$. In turn, we say $B_{u'} \subset B_{u}$ if $u' \subset u$.

\begin{defi}[Independence/completeness]
We say that the branches $B_{u}$ and $B_{u'}$ are {\em independent} if $u\not \subset u'$ and $u' \not \subset u$. We say that a set of branches of $C_{n}$ is {\em complete} if their union is equal to $C_{n}$. 
\end{defi}

Suppose now that $\mathcal  B= \{B_{u_i}\}$ is a finite, complete set of pairwise independent branches of $C_{n}$. Observe that each $u_i$ determines a unique geodesic on $\mathcal{T}_n$, with the root as one of its endpoints. Further, the branches being pairwise independent is equivalent to the fact that $u_i$ does not lie on the geodesic determined by $u_j$, with $i\ne j$. In this way, $\mathcal B$ uniquely determines a finite rooted tree $\mathcal{T}_\mathcal{B} \subset \mathcal{T}_{n}$, whose root is the root of $\mathcal{T}_{n}$, and whose leaves are naturally labelled by the $u_i$'s. Finally, observe that since the set $\mathcal{B}$ is complete, every vertex of $\mathcal{T}_\mathcal{B}$ which is neither the root nor a leaf  has exactly $n+1$ neighbours. 

\subsection{Expansions}  
Let $B_u$ be a branch of $C_n$, where $u$ is a finite word in $\{1, \ldots, n\}$. The {\em elementary expansion} of $B_u$ is the set of branches \linebreak $\{B_{u1}, B_{u2},\ldots, B_{un}\}$. 
Intuitively, the elementary expansion of $B_u$ corresponds to the result of adding the $n$ descendants of $u$ in $\mathcal{T}_n$ to the geodesic from the root of $\mathcal{T}_n$ to $u$. Conversely, an {\em elementary reduction} of $\{B_{u1}, B_{u2},\ldots, B_{un}\}$ is $B_u$.   

\begin{defi}[Elementary expansion/reduction]
\label{def:elementary}
Let $\mathcal{B}$  be a complete set of
pairwise independent branches of $C_{n}$.
An {\em elementary expansion} (resp. {\em elementary reduction}) of $\mathcal{B}$ is a set $\mathcal{B}'$ of branches which is obtained from $\mathcal{B}$ by elementary expansions (resp. elementary reductions) on some of the branches of $\mathcal{B}$.
\end{defi}

Observe that, in definition \ref{def:elementary}, the set $\mathcal{B}'$ is also a complete set of pairwise independent branches of $C_{n}$. 

\begin{defi}[Expansion/reduction]
Let $\mathcal B$ and $\mathcal B'$ be two complete sets of pairwise independent branches of $C_{n}$. We say that $\mathcal{B}'$ is an {\em expansion} (resp. {\em reduction}) of $\mathcal B$ if $\mathcal B '$ may be obtained from $\mathcal B$ by a finite sequence of elementary expansions (resp. elementary reductions). A complete set of pairwise independent branches is \textit{reduced} if no more elementary reductions can be performed on it.
\end{defi}

In light of the equivalence between sets of branches and finite rooted trees, in what follows we will also speak about the {\em elementary expansion/reduction of a tree} and a {\em reduced tree}. Observe that, given two complete sets of pairwise independent branches $\mathcal B$ and $\mathcal B '$, their intersection $\mathcal{B} \cap \mathcal B'$ is an expansion of both, which is again a complete set of pairwise independent branches. In particular, we have proved the following fact, which will be crucial in what follows: 

\begin{lemma}
Any two complete sets of
pairwise independent branches of $C_{n}$ have a common expansion. 
\label{lem:commonexpansion}
\end{lemma}

\subsection{Self-similar bijections}
For convenience, we will use the following notation. Let $\sigma \in \Sym(n)$ be any element of the symmetric group of $n$ elements. Given any word $w = w_1w_2w_3 \dots \in \{1, \ldots, n\}^\mathbb{N}$ we define $\sigma(w) = \sigma(w_1)\sigma(w_2)\sigma(w_3)\dots \in \{1, \ldots, n\}^\mathbb{N}$.

\begin{defi}[Self-similar bijection]
\label{def:selfsimilar}
Let $u, u'$ be two finite words in $\{1, \ldots, n\}$. We say that a bijection $g:B_{u}\to B_{u'}$ is {\em self-similar} if there exists $\sigma \in \Sym(n)$ such that $g(uw)=u'\sigma(w)$ for every $w \in \{1, \ldots, n\}^\mathbb{N}$.
\end{defi}

Fix an integer $n\ge 2$, a subgroup $H \leq \Sym(n)$ and consider the set $V_n(H)$ of all bijections $g:C_n \to C_n$ for which there exists  a finite complete set $\mathcal B = \mathcal B(g)$ of pairwise independent branches such that, $\forall B\in \mathcal B$, we have:

\begin{enumerate}
\item The set $g(B)$ is a branch, and
\item The restriction 
$g_{\mid B}: B \to g(B)$ is self-similar. 
\end{enumerate} 

Let $B_u = B, g$ and $\sigma$ as in Definition \ref{def:selfsimilar}. Then $g$ acts on the elementary expansion of $B$ by permuting its leaves according to $\sigma$, that is, taking $B_i$ to $B_{\sigma(i)} \forall i \in \{1, \dots, n\}$. Conversely, the elementary reduction of $\{B_{1}, \dots,B_{n}\}$ to $B$ is only possible if $g$ permutes its leaves according to $\sigma$, that is, taking $B_i$ to $B_{\sigma(i)}\forall i \in \{1, \dots, n\}$. We record the following observation:

\begin{lemma}
Let $g:C_n\to C_n$ be a bijection and $\mathcal{B}$ a set of pairwise independent branches satisfying (1) and (2) above. If $\mathcal{B}'$ is an expansion of $\mathcal{B}$ then $\mathcal{B}'$ also satisfies (1) and (2).
\label{lem:expansions}
\end{lemma}

We have the following useful consequence:

\begin{corollary}
\label{cor:welldef}
Let $g: C_n \to C_n$ be a bijection satisfying (1) and (2). As a transformation of $C_n$, $g$ is independent of the choice of set of branches $\mathcal{B}(g)$. 
\end{corollary}

Now, we proceed to prove that the set $V_n(H)$ is, indeed, a group:

\begin{prop}\label{comp}
$V_n(H)$ is a group under composition. 
\end{prop}

\begin{proof}
The only non-straightforward part of the proof is to show that $V_n(H)$ is closed under composition.
Let $f,g \in V_n(H)$ and $\mathcal{B}$ (resp. $\mathcal{B}'$) a complete set of pairwise independent branches associated to $f$ (resp. $g$). By Lemma \ref{lem:commonexpansion}, $f(\mathcal{B})$ and $\mathcal{B}'$ have a common expansion, say $\mathcal{B}''$. Now, by Corollary \ref{cor:welldef}, we may represent $f$ (resp. $g$) using the set $f^{-1}(\mathcal{B}'')$ (resp. $\mathcal{B}'')$. So, for any word $w \in \{1, \ldots, n\}^\mathbb{N}$: 
$$g \circ f = g(f(uw))= f(u'\sigma(w))=  u''(\sigma' \circ \sigma)(w).$$
Thus $g \circ f$ is another element of $V_n(H)$. 
\end{proof}

Using the incarnation of $C_n$ as the set of geodesics in $\mathcal{T}_n$ issued from the root, plus the fact that elements of $V_n(H)$ are self-similar outside a finite subtree of $\mathcal{T}_n$, we immediately obtain: 

\begin{lemma}
For every $n\ge 2$, $H \leq \Sym(n)$, the group $V_n(H)$ is a subgroup of $\Homeo(C_n)$. 
\end{lemma}

\subsection{Tree-pair representation} Following \cite{CF}, we now explain a useful way of encoding elements of $V_{n}(H)$ in terms of pairs of trees. Given $g\in V_{n}(H)$, let $\mathcal{B}$ and $g(\mathcal{B})$ be two complete sets of pairwise independent branches which satisfy (1) and (2); in light of Corollary \ref{cor:welldef}, the action of $g$ on $C_n$ is independent of this choice. 

As previously explained, $\mathcal{B}$ and $g(\mathcal{B})$ each determine a unique finite tree, denoted $\mathcal{T}_\mathcal{B}$ and $\mathcal{T}_{g(\mathcal{B})}$ respectively. Observe that $g$ induces a bijective map $\tau = \tau(g)$ from the set of leaves of $\mathcal{T}_\mathcal{B}$ to that of $\mathcal{T}_{g(\mathcal{B})}$; in particular, $\tau$ may be regarded as an element of the permutation group $\Sym(k)$ on $k$ elements, where $k$ is the number of leaves of $\mathcal{T}_\mathcal{B}$. Finally, for every leaf of $\mathcal{T}_{g(\mathcal{B})}$ we have a label $\sigma \in H$ imposing how $g$ acts on every independent branch of $g(\mathcal{B})$. These $k$ labels are expressed as an element $\sigma = (\sigma_1, \dots, \sigma_{k}) \in H^k$. 

Summarizing, we have encoded the action of $g$ by $(\mathcal{T}_\mathcal{B},\mathcal{T}_{g(\mathcal{B})}, \tau, \sigma)$. Of course, this way of encoding is far from unique: we say that $g$ is \textit{reduced} if $\mathcal{B}$ and $g(\mathcal{B})$ cannot be reduced any further. Finally, we will simply write $(\mathcal{T},\mathcal{T}',\tau,\sigma)$ for a tree-pair representative of $g$ whenever it is not necessary to specify the particular branches giving rise to the trees. 


\section{Strand diagrams}

In this section we introduce $n$-strand diagrams, which are a family of directed graphs with labelled vertices. They are a generalization of the strand diagrams presented in \cite{BM}, which have no labels. As we will see, $n$-strand diagrams are able to represent elements of any Thompson-like group. Once they are defined, we establish an equivalence relation between them which will turn to be the key to solve the conjugacy problem. 

\subsection{Basics about graphs} Let $\Gamma$ be a directed graph. We will assume that the vertices of $\Gamma$ are labelled by elements of $\Sym(n)$, so let $V(\Gamma) = \linebreak \{v_1^{\sigma_1}, v_2^{\sigma_2},\dots, v_k^{\sigma_k}\}$ be the set of vertices $v_i$ with label $\sigma_i \in \{\emptyset\}  \cup \Sym(n) \backslash Id$, for some $n \geq 2$. When $\sigma_i = \emptyset$, the corresponding vertex has no labels. This notation will only be used in this subsection, as we need to distinguish labelled and nonlabelled vertices, so we will simply write $v$ when the label is not relevant for the definitions or proofs. Let $E(\Gamma) = \{e_1,e_2,\dots, e_s\} \subset V(\Gamma) \times V(\Gamma)$ be the set of edges, which have the form $e = (v_i^{\sigma_i},v_j^{\sigma_j})$. An edge $e = (v,v')$ is oriented from $v$ to $v'$. An \textit{oriented path} is a sequence of edges $\{e_1,...,e_t\} = \{(v_{i(1)},v_{j(1)}),\dots, (v_{i(t)},v_{j(t)})\}$ such that $v_{j(k)} = v_{i(k+1)} \, \forall k \in \{1,\dots, t-1\}$. In addition, if $v_{j(t)} = v_{i(1)}$, we have an \textit{oriented loop}. We say that $\Gamma$ is \textit{acyclic} if it has no oriented loops.

\begin{defi}[Degree]
The \textit{degree} of a vertex $v \in V(\Gamma)$ is the number of edges which have $v$ as endpoint, that is, those which have the form $(v,v')$ or $(v',v)$ for some $v' \in V(\Gamma)$. If some edge has the form $(v,v)$, it is counted twice.
\end{defi}

\begin{defi}[Source/sink]
A vertex $v$ is a \textit{source} (resp. a \textit{sink}) for a finite set of directed edges if they have $v$ as start point (resp. endpoint). An empty-labelled vertex $v^{\emptyset}$ is a \textit{main source} (resp. a \textit{main sink}) if it is a source (resp. a sink) of degree one. A graph $\Gamma$ is a \textit{digraph} if it has exactly one main source and one main sink. 
\end{defi}

\begin{defi}[$n$-split/$n$-merge/$\sigma$-vertex]
Let $n \geq 2$. An \textit{$n$-split} (resp. an \textit{$n$-merge}) is an empty-labelled vertex $v^{\emptyset}$ of degree $n+1$ which is a sink for one edge and a source for the others (resp. a source for one edge and a sink for the others).
A \textit{$\sigma$-vertex} is a labelled vertex $v^{\sigma}$ of degree $2$ which is a sink for one edge and a source for the other. 
\end{defi}

In the graphs depicted on this paper, $n$-splits and $n$-merges will be represented as black vertices with no label, while $\sigma$-vertices will be represented as white vertices with the corresponding label; see Figure \ref{fig:edgeorder}. 

\begin{defi}[Pitchfork graph]
A \textit{pitchfork graph} is defined as any graph whose vertices are only main sources, main sinks, $n$-splits, $n$-merges or $\sigma$-vertices. 
\end{defi}

\subsection{Strand diagrams} 
Let $g \in V_n(H)$ for some $n \geq 2$, $H \leq \Sym(n)$, and $(\mathcal{T},\mathcal{T}', \tau, \sigma)$ a tree-pair representative of $g$. We build a graph from $g$ linking the leaves of $\mathcal{T}$ and $\mathcal{T}'$ as stated by $\tau$ and appending to the roots of both $\mathcal{T}$ and $\mathcal{T}'$ an edge and a vertex. See Figure \ref{fig:construction}. We give the orientation from the vertex appended in $\mathcal{T}$ to the leaves of $\mathcal{T}$, from the leaves of $\mathcal{T}$ to the leaves of $\mathcal{T}'$ and finally from them to the appended vertex of $\mathcal{T}'$. In this way, we obtain a directed acyclic labelled pitchfork digraph  called the \textit{$n$-strand diagram} of $g$.

If $\Gamma$ is not planar, that is, it cannot be embedded in the plane, there must be \textit{crossings} between some edges. In other words, two edges intersect in a point which is not a vertex of $\Gamma$. To distinguish isomorphic graphs with different crossings, we impose a rotation system:
\begin{defi}[Rotation system]
Let $\Gamma$ be a pitchfork graph. A \textit{rotation system} of $\Gamma$ is a map $\rho_{\Gamma}: E(\Gamma) \longrightarrow \{0, \dots, n\}^2$ which gives an order to every edge of $\Gamma$ around its endpoints as follows:
\begin{enumerate}
\item A counterclockwise order to the directed edges of an $n$-split $v$, where the $0$-th edge is the edge which has $v$ as sink.
\item A clockwise order to the directed edges of an $n$-merge $v$, where the $0$-th edge is the edge which has $v$ as source.
\item A $0$ to the edge which has a $\sigma$-vertex $v$ as sink, and a $1$ to the one which has $v$ as source.
\end{enumerate}
\end{defi}
It follows that every edge of $\Gamma$ has a rotation pair assigned depending on the labels of its endpoints; see Figure \ref{fig:edgeorder}. Observe that $\rho_{\Gamma}$ completely determines the crossings of $\Gamma$.

\begin{figure}[h]
\labellist
\pinlabel $0$ at 12 46
\pinlabel $0$ at 56 17 
\pinlabel $1$ at 47 60
\pinlabel $2$ at 56 66
\pinlabel $n-1$ at 71 66
\pinlabel $n$ at 80 60
\pinlabel $1$ at 3 6
\pinlabel $2$ at 12 0
\pinlabel $n-1$ at 27 0
\pinlabel $n$ at 38 6
\pinlabel $0$ at 95 46 
\pinlabel $1$ at 95 17
\pinlabel $\sigma$ at 110 35
\endlabellist
\centering
\includegraphics[width=0.4\textwidth]{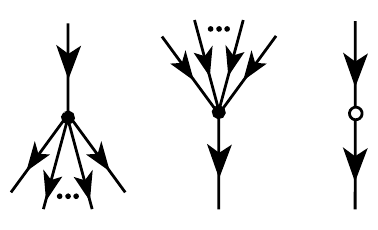}
\caption{An $n$-split, an $n$-merge and a $\sigma$-vertex with the rotation system $\rho$.}
\label{fig:edgeorder}
\end{figure}
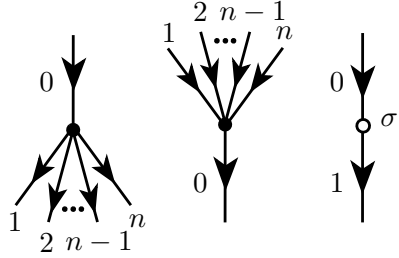

We obtain an $n$-strand diagram as a consequence of the process described at the beginning of this subsection, but not all $n$-strand diagrams are obtained in this way. We proceed to give the definition:

\begin{defi}[$n$-strand diagram]
\label{def:strand}
An \textit{$n$-strand diagram} is a finite labelled directed acyclic pitchfork digraph $\Gamma$. Two $n$-strand diagrams $\Gamma$ and $\Gamma'$ are equal if there exists an isomorphism $\phi: \Gamma \rightarrow \Gamma'$ such that:
\begin{enumerate}
\item $\phi(\rho_{\Gamma}) = \rho_{\Gamma'}$, and
\item $\phi(v^{\sigma}) = \phi(v)^{\sigma}$, $\forall v \in V(\Gamma)$.
\end{enumerate} 
\end{defi}

\subsection{Closed strand diagrams}
Let $\Gamma$ be an $n$-strand diagram. Suppose that we identify its main source and main sink, and erase the resulting vertex, as Figure \ref{fig:construction} shows. Let $e$ be the new edge created by this identification. Then $\Gamma$ is naturally immersed in an annulus  as follows: let $c(\alpha) \in H^1(\Gamma, \mathbb{Z})$ be the counterclockwise winding number of a directed loop $\alpha$ around the central hole. Then we impose that $c(\alpha_i) = 1$ for all directed loops $\alpha_i \subset \Gamma$ containing $e$. We denote as $\overline{\Gamma}$ the closure of an $n$-strand diagram $\Gamma$ by this process. Again, not all closed $n$-strand diagrams are obtained in this way. 

\begin{defi}[Closed $n$-strand diagram]
A \textit{closed $n$-strand diagram} is a finite labelled directed pitchfork graph without main sources and main sinks, where $c(\alpha) >0$ for every oriented loop $\alpha \subset \Gamma$. Two closed $n$-strand diagrams $\Gamma$ and $\Gamma'$ are equal if there exists an isomorphism $\phi: \Gamma \rightarrow \Gamma'$ such that:
\begin{enumerate}
\item $\phi(\rho_{\Gamma}) = \rho_{\Gamma'}$,
\item $\phi(v^{\sigma}) = \phi(v)^{\sigma}$, $\forall v \in V(\Gamma)$, and
\item $c(\alpha) = c(\phi(\alpha))$ for every oriented loop $\alpha \subset \Gamma$.
\end{enumerate}
\end{defi}

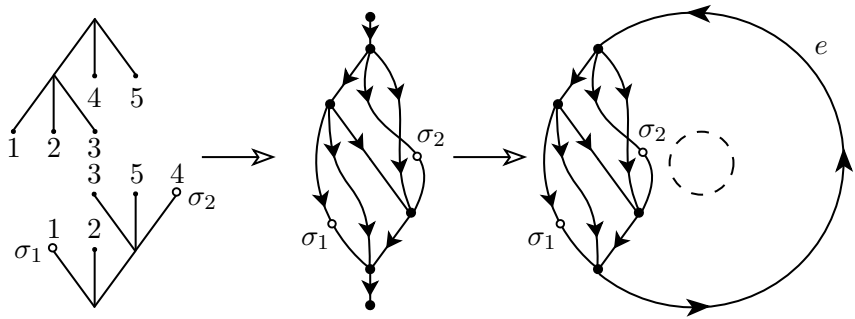
\begin{figure}[h]
\centering
\labellist
\pinlabel $1$ at 4 70
\pinlabel $1$ at 20 38
\pinlabel $2$ at 20 70
\pinlabel $2$ at 36 38
\pinlabel $3$ at 36 70
\pinlabel $3$ at 36 60
\pinlabel $4$ at 36 90
\pinlabel $5$ at 53 60
\pinlabel $5$ at 53 90
\pinlabel $4$ at 69 60
\pinlabel $\sigma_1$ at 10 28
\pinlabel $\sigma_2$ at 79 49
\pinlabel $\sigma_1$ at 214 33
\pinlabel $\sigma_2$ at 258 76
\pinlabel $\sigma_1$ at 124 33
\pinlabel $\sigma_2$ at 170 74
\pinlabel $e$ at 325 110
\endlabellist
\includegraphics[width=0.9\textwidth]{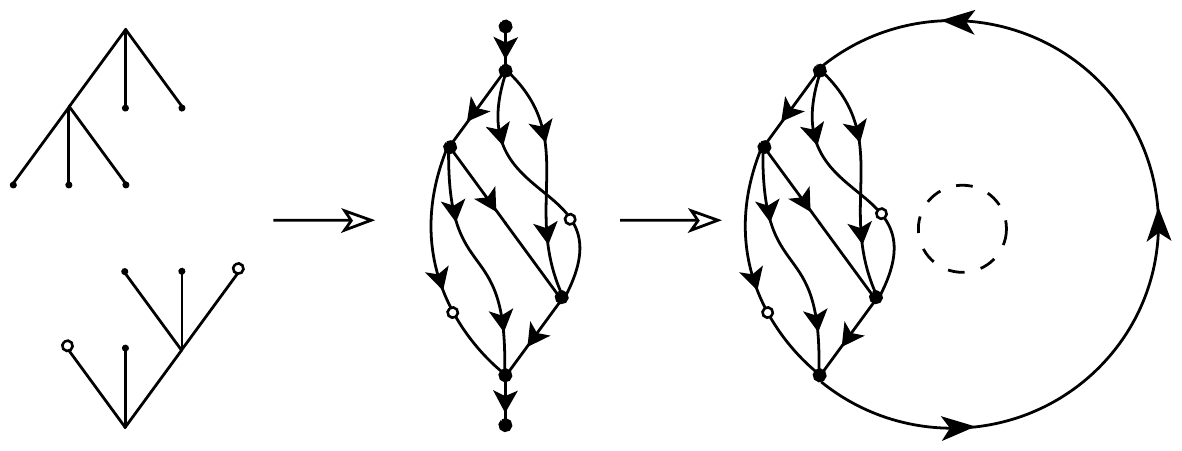}
\caption{A $3$-strand diagram and a closed $3$-strand diagram.}
\label{fig:construction}
\end{figure} 

\begin{remark}
\label{rem:samenumber}
Observe that an $n$-strand diagram is always connected, whilst a closed $n$-strand diagram do not need to be connected. In addition, every (closed) $n$-strand $\Gamma$ diagram has the same number of $n$-splits and $n$-merges. 
\end{remark}

\subsection{Equivalence relations between strand diagrams} We now define the set of transformations for modifying any (closed) $n$-strand diagram $\Gamma$. Let $v_1$, $v_2 \in V(\Gamma)$, we denote the \textit{subgraph between $v_1$ and $v_2$} to the set of all oriented paths which start at $v_1$ and end in $v_2$. The \textit{support} of all transformations will be $\varepsilon$-neighbourhoods of these subgraphs, assuming that $\Gamma$ is undirected and all edges have length 1. 

\begin{defi}[Type I reduction] Let $\sigma \in \Sym(n)$ be fixed. Let $\Gamma_0$ be the subgraph between an $n$-split $v_1$ and an $n$-merge $v_2$ such that the $i$-th edge of $v_1$ and the $\sigma(i)$-th edge of $v_2$ have the same $\sigma$-vertex as endpoint, $\forall i \in \{1,\dots,n\}$. A \textit{type I reduction} is a transformation that replaces the $\varepsilon$-neighbourhood of $\Gamma_0$ with a $\sigma$-vertex. See Figure \ref{fig:conred}.

If $\Gamma_0$ is the subgraph between an $n$-split $v_1$ and an $n$-merge $v_2$ such that the $i$-th edge of $v_1$ is the $i$-th edge of $v_2$ $\forall i \in \{1,\dots,n\}$, an \textit{identity type I reduction} replaces its $\varepsilon$-neighbourhood with an edge.
\end{defi} 

\begin{figure}[h]
\labellist
\pinlabel $\tau$ at 96 60
\pinlabel $\tau$ at 146 60
\pinlabel $\tau$ at 186 50
\endlabellist
\centering
\includegraphics[width=0.6\textwidth]{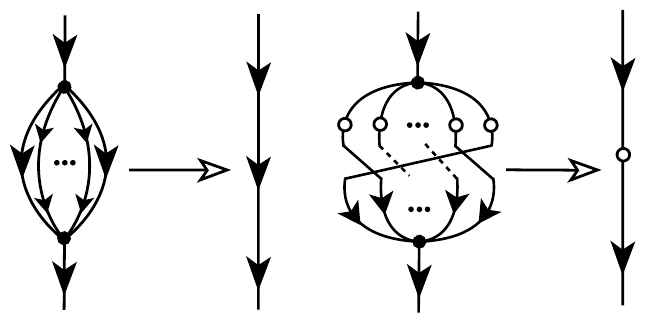}
\caption{An identity type I reduction and a type I reduction with $\tau$-vertices, where $\tau$ takes $i$ to $i+1\mod n$.}
\label{fig:conred}
\end{figure}

\begin{defi}[Type II reduction]
Let $\sigma \in \Sym(n)$ be fixed. Consider the $\varepsilon$-neighbourhood of the subgraph $\Gamma_0$ between an $n$-merge $v_1$ and an $n$-split $v_2$ whose $0$-th edges have the same $\sigma$-vertex as endpoint. A \textit{type II reduction} replaces the $\varepsilon$-neighbourhood of $\Gamma_0$ with the $\varepsilon$-neighbourhood of $n$ $\sigma$-vertices such that the $i$-th edge of $v_1$ and the $\sigma(i)$-th edge of $v_2$ have the same $\sigma$-vertex as endpoint, $\forall i \in \{1,\dots,n\}$. See Figure \ref{fig:expred}.

If $\Gamma_0$ is the subgraph between an $n$-merge $v_1$ and an $n$-split $v_2$ such that the $0$-th edge of $v_1$ is the $0$-th edge of $v_2$, an \textit{identity type II reduction} replaces the $\varepsilon$-neighbourhood of $\Gamma_0$ with a set of $n$ edges such that the $i$-th edge of $v_1$ is the $i$-th edge of $v_2$, $\forall i \in \{1,\dots,n\}$.
\end{defi}

\begin{figure}[h]
\labellist
\pinlabel $\sigma$ at 163 60
\pinlabel $\sigma$ at 232 50
\endlabellist
\centering
\includegraphics[width=0.7\textwidth]{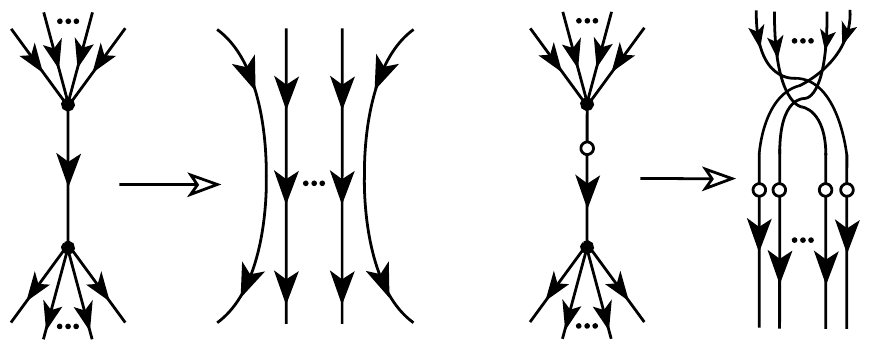}
\caption{An identity type II reduction and a type II reduction with $\sigma$-vertices, where $\sigma$ takes $i$ to $-i+1 \mod n$.}
\label{fig:expred}
\end{figure}

\begin{defi}[Type III reduction] A \textit{type III reduction} replaces the $\varepsilon$-neighbourhood of the subgraph $\Gamma_0$ consisting of one edge whose endpoints are a $\sigma_1$-vertex and a $\sigma_2$-vertex, with a $(\sigma_2 \circ \sigma_1)$-vertex. If $\sigma_2 \circ \sigma_1 = Id$, it replaces the $\varepsilon$-neighbourhood of $\Gamma_0$ with an edge.
\end{defi}

\begin{remark}
\label{rem:inv}
Type I and II reductions have an inverse, since they are uniquely determined by the $\sigma$-vertex involved (or the identity). However, type III reductions do not have an inverse, since the same element $\sigma \in \Sym(n)$ can be obtained composing different pairs of elements.
\end{remark}

\begin{defi}[Free loop]
Let $\Gamma$ be a closed $n$-strand diagram.  A \textit{free loop} of $\Gamma$ is an oriented loop which has neither $n$-splits nor $n$-merges.
\end{defi}

The following reduction is a composition of type I and type II reductions and their inverses. We define it for convenience, as we will see in Proposition \ref{prop:equiv}.

\begin{defi}[Type IV reduction] Let $\Gamma_0$ be the $\varepsilon$-neighbourhood of the subgraph consisting of an edge $e$ whose endpoints are a $n$-merge and a $\sigma$-vertex, where $e$ is the $0$-th edge of the $n$-merge. Consider $\Gamma_1$, obtained from $\Gamma_0$ by performing the inverse of a type I reduction on the $\sigma$-vertex. Let $\Gamma_2$ be obtained from $\Gamma_1$ by performing a type II reduction on the $n$-merge and the new $n$-split created by the previous reduction. A \textit{type IV reduction} replaces $\Gamma_0$ with $\Gamma_2$. The symmetric case (an $n$-split whose $0$-th edge has a $\sigma$-vertex as endpoint) is treated in a similar way.

We can also perform a type IV reduction when $\Gamma_0$ is a free loop containing only one $\sigma$-vertex, since the inverse of a type I reduction creates an $n$-split and an $n$-merge which share the same $0$-th edge. In this case the type IV reduction replaces $\Gamma_2$ by $\Gamma_0$. By Remark \ref{rem:inv} this reduction is well defined. All cases are depicted on Figure \ref{fig:brired}.
\end{defi}

\begin{figure}[h]
\labellist
\pinlabel $\Gamma_2$ at 40 140
\pinlabel $\Gamma_1$ at 105 145
\pinlabel $\Gamma_0$ at 135 150
\pinlabel $\Gamma_2$ at 205 200
\pinlabel $\Gamma_1$ at 270 210
\pinlabel $\Gamma_0$ at 305 190
\pinlabel $\Gamma_2$ at 90 30
\pinlabel $\Gamma_1$ at 210 30
\pinlabel $\Gamma_0$ at 330 30
\endlabellist
\centering
\includegraphics[width=0.85\textwidth]{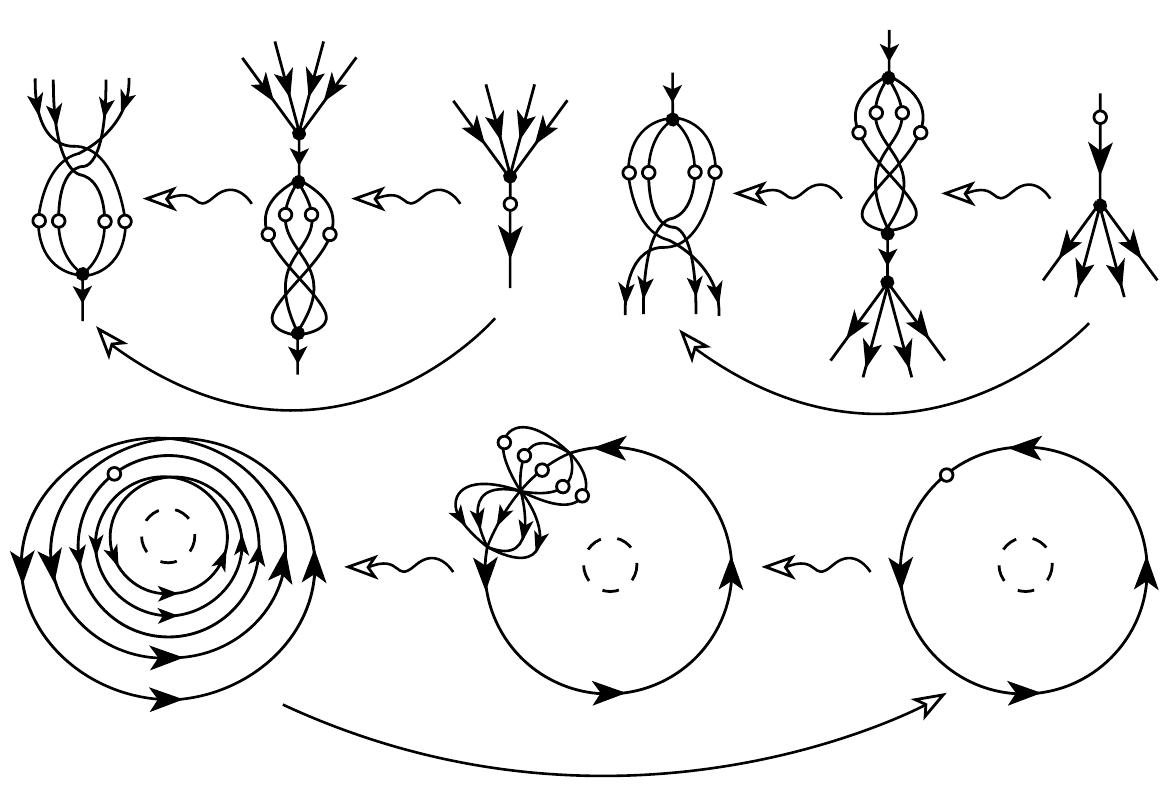}
\caption{Type IV reductions for $n =4, 5$ with $\sigma$-vertices, where $\sigma$ takes $i$ to $-i+1 \mod n$.}
\label{fig:brired}
\end{figure}
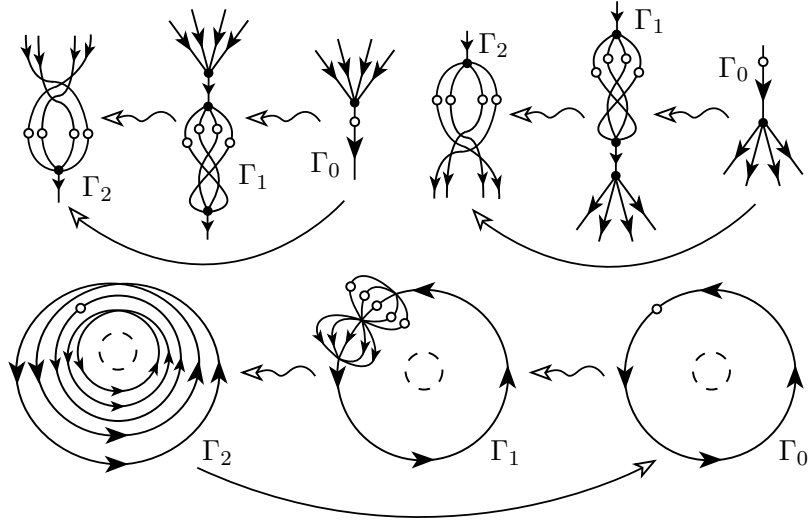

\begin{defi}[Conjugating transformation]
Let $\Gamma_0$ be a subgraph consisting of a free loop containing a sequence of $\sigma_i$-vertices, $i \in \{1,\dots,n\}$ which occurs in cyclic order. Let $\Gamma_i$ be the subgraph obtained from $\Gamma_0$ by performing type III reductions until we have a $(\sigma_{i-1} \circ \dots \circ \sigma_1 \circ \sigma_n \circ \dots \circ \sigma_i)$-vertex, for some $i \in \{1,\dots,n\}$. We define a \textit{conjugating transformation} as an exchange from $\Gamma_i$ to $\Gamma_j$ $\forall i,j \in \{1,\dots,n\}$. See Figure \ref{fig:conj}.
\end{defi}

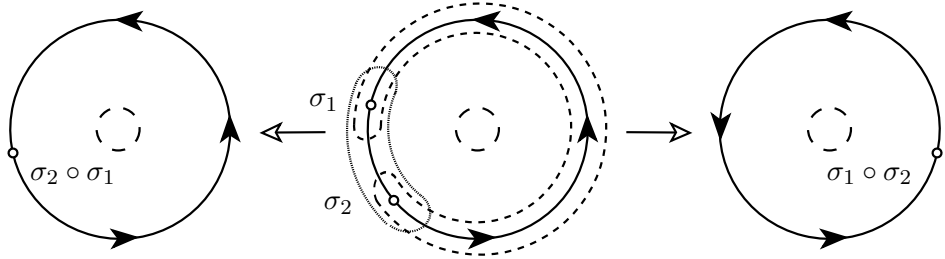
\begin{figure}[h]
\labellist
\pinlabel $\sigma_1$ at 110 55
\pinlabel $\sigma_2$ at 115 20
\pinlabel $\sigma_2\circ \sigma_1$ at 25 30
\pinlabel $\sigma_1\circ \sigma_2$ at 295 30
\endlabellist
\centering
\includegraphics[width=1\textwidth]{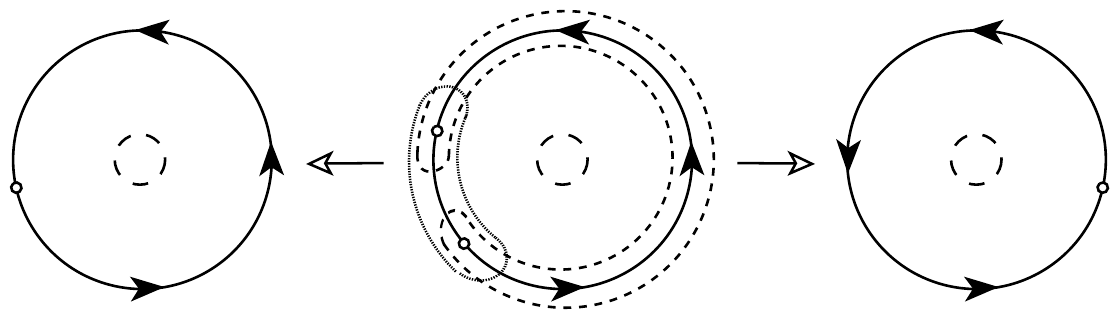}
\caption{A conjugating transformation with two $\sigma$-vertices.}
\label{fig:conj}
\end{figure}

\begin{remark}
\label{rem:conjugatingtransformation} 
Observe that if all $\sigma_i \in H$, then all $\sigma_{i-1} \circ \dots \circ \sigma_1 \circ \sigma_n \circ \dots \circ \sigma_i$ are conjugate in $H$. Thus a conjugating transformation does not change the conjugacy class of the $\sigma$-vertex obtained at the end.
\end{remark}

\begin{defi}[Equivalence]
\label{def:equiv}
Two $n$-strand diagrams are \textit{equivalent} if we can obtain one from the other by performing a finite sequence of type I, II, III and IV reductions and their inverses. 
\end{defi}

\begin{defi}[Reduction]
\label{def:red}
An $n$-strand diagram is \textit{reduced} if no type I, II, III and IV reductions can be performed on it. 
\end{defi}

\begin{prop}
\label{prop:equiv}
Every $n$-strand diagram is equivalent to a unique reduced $n$-strand diagram.
\end{prop}

\begin{proof}
Firstly, observe that the number of $n$-splits, $n$-merges and $\sigma$-vertices of any $n$-strand diagram is finite. Type I and II reductions decrease the number of $n$-splits and $n$-merges and type III and IV reductions do not change it. For the same reason, the number of type IV reductions is finite, so the number of type III reductions is also finite. Thus the process of reducing is finite. 

Secondly, reductions are locally confluent, that is,  suppose that we can perform two different reductions $r_0$, $r_1$ on an $n$-strand diagram $\Gamma$. Then there exist two sequences of reductions $s_0$, $s_1$ such that $s_1 \circ r_1(\Gamma) = s_0 \circ r_0(\Gamma)$. Let $\Gamma_i$ be the support of $r_i$. If $\Gamma_0 \cap \Gamma_1 = \emptyset$, then $r_0$ and $r_1$ commute, that is, $r_0 \circ r_1(\Gamma) = r_1 \circ r_0(\Gamma)$. If $\Gamma_0 \cap \Gamma_1 \neq \emptyset$, then one of the $r_i$ is a type I reduction and the other a type II reduction; or both are type III reductions:
\begin{enumerate}
\item If $r_0$ is an identity type I reduction and $r_1$ is an identity type II reduction, then $r_0(\Gamma) =r_1(\Gamma)$, as Figure \ref{fig:equiv} shows.
\item If $r_0$ is a type II reduction with a $\sigma$-vertex and $r_1$ is an identity type I reduction, then there exists a type IV reduction $b$ with support $\Gamma_0 \cup \Gamma_1$ such that $r_0(\Gamma) = b \circ r_1(\Gamma)$. The same occurs for the symmetric case where $r_0$ is an identity type II reduction and $r_1$ is a type I reduction with $\sigma$-vertices.
\item If $r_0$ is a type II reduction with a $\sigma$-vertex and $r_1$ is a type I reduction with a $\sigma'$-vertex, then there exists a type IV reduction $b$ with support $\Gamma_0 \cup \Gamma_1$ and type III reductions $m_0,\dots,m_n$ such that $b \circ m_{n-1} \circ \dots\circ m_0 \circ r_0(\Gamma) = m_n \circ r_1(\Gamma)$.
\item If $r_0$ and $r_1$ are both type III reductions, there exist two type III reductions $r'_0, r'_1$ such that $r'_0 \circ r_0(\Gamma) =  r'_1\circ r_0(\Gamma)$, as  in Figure \ref{fig:equiv}.
\end{enumerate}
As the process of reduction finishes and the reductions are locally confluent, we have a unique reduced strand $n$-diagram for each equivalence class.\end{proof}

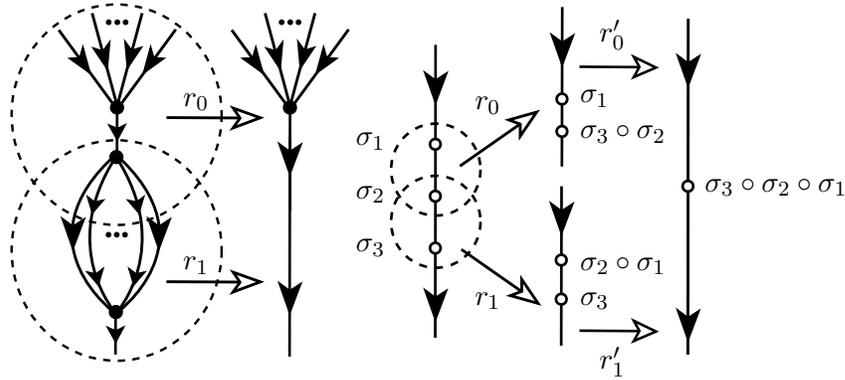
\begin{figure}[h]
\labellist
\pinlabel $\sigma_1$ at 167 77
\pinlabel $\sigma_3\circ\sigma_2$ at 175 67
\pinlabel $\sigma_2\circ\sigma_1$ at 175 30
\pinlabel $\sigma_3$ at 167 20
\pinlabel $\sigma_3\circ\sigma_2\circ\sigma_1$ at 218 52
\pinlabel $r_0$ at 137 75
\pinlabel $r_1$ at 137 20
\pinlabel $r'_0$ at 172 95
\pinlabel $r'_1$ at 172 3
\pinlabel $\sigma_1$ at 104 65
\pinlabel $\sigma_2$ at 104 50
\pinlabel $\sigma_3$ at 104 35
\pinlabel $r_0$ at 55 76
\pinlabel $r_1$ at 55 30
\endlabellist
\centering
\includegraphics[width=0.8\textwidth]{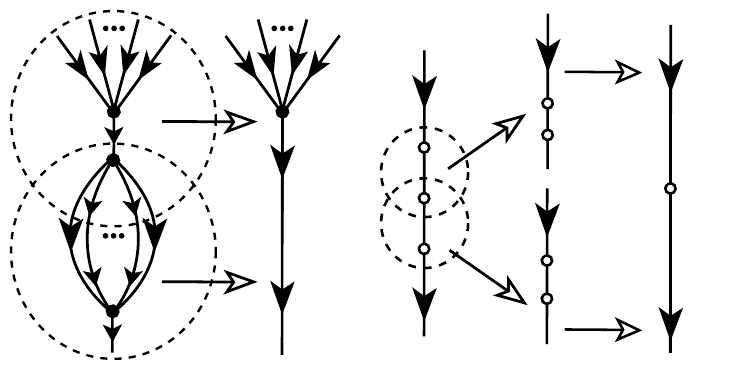}
\caption{The confluences (1) and (4).}
\label{fig:equiv}
\end{figure}

As we have realized on Proposition \ref{prop:equiv}, type III reductions must be considered in order to ensure the local confluence. 

\begin{prop}
There is a bijection between reduced $n$-strand diagrams and reduced elements of $V_n(H)$.
\end{prop}

\begin{proof}
On the one hand, we always obtain a reduced $n$-strand diagram $\Gamma^g$ from a reduced element $g \in V_n(H)$, as showed in Subsection 3.2. On the other hand, let $\Gamma$ be any reduced $n$-strand diagram. Then every oriented path from the main source to the main sink consists of a finite sequence of vertices which are $n$-splits, $n$-merges or $\sigma$-vertices. We have the following facts, bearing in mind that $\Gamma$ is reduced:
\begin{enumerate}
\item There are no two (or more) consecutive $\sigma$-vertices, otherwise we could perform a type III reduction on $\Gamma$. 
\item There are no $n$-merges followed by an $n$-split, or an $n$-merge followed by a $\sigma$-vertex and then by an $n$-split, otherwise we could perform a type II reduction on $\Gamma$. 
\item There are no $\sigma$-vertices before any $n$-split or after any $n$-merge, otherwise we could perform a type IV reduction on $\Gamma$. 
\end{enumerate}
Therefore, every oriented path of $\Gamma$ consists of a sequence of a finite number of $n$-splits followed by at most one $\sigma$-vertex and the same finite number of $n$-merges. We then `cut' $\Gamma$, cutting everyone of these paths in one edge between the set of $n$-splits and the set of $n$-merges, obtaining a tree-pair representing an element $g \in V_n(H)$. Both trees of $g$ have the same number of leaves by Remark \ref{rem:samenumber}. Finally, $g$ must be reduced: otherwise a type I reduction could be performed on $\Gamma$. 
\end{proof}


\section{The conjugacy problem}

In this section, we prove Theorem \ref{MT}, which will be an immediate consequence of Theorem \ref{TSD} below. Theorem \ref{TSD} and the argument for proving it are adaptations of \cite[Theorem 3.1]{BM} to (closed) $n$-strand diagrams.

\begin{defi}[$(p,q,n)$-strand diagram] A $(p,q,n)$-strand diagram is a finite labelled directed acyclic pitchfork graph with $p$ main sources and $q$ main sinks together with a rotation system $\rho$. The conditions for two $(p,q,n)$-strand diagrams to be equal are analogous as in Definition \ref{def:strand}.
\end{defi}
In the case of $(p,q,n)$-strand diagrams, the rotation system $\rho$ imposes an order from left to right to its main sources and its main sinks, in order to distinguish crossings between them, as Figure \ref{fig:pqstrand} shows.
\begin{figure}[h]
\labellist
\pinlabel $1$ at 28.5 137
\pinlabel $2$ at 74.5 137
\pinlabel $3$ at 121.5 137
\pinlabel $1$ at 29 -3
\pinlabel $3$ at 75 -3
\pinlabel $5$ at 121 -3
\pinlabel $2$ at 56 -3
\pinlabel $4$ at 93.5 -3
\endlabellist
\centering
\includegraphics[width=0.35\textwidth]{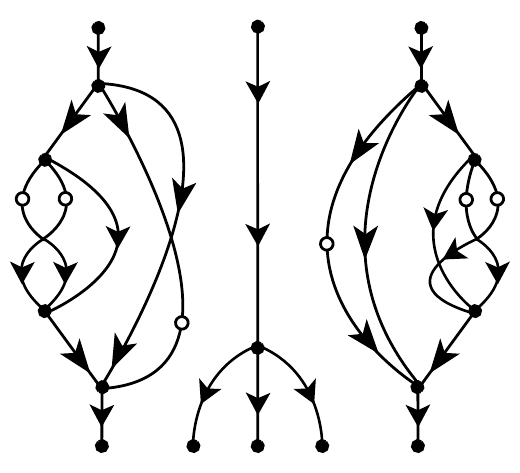}
\caption{An example of $(3,5,3)$-strand diagram.}
\label{fig:pqstrand}
\end{figure}
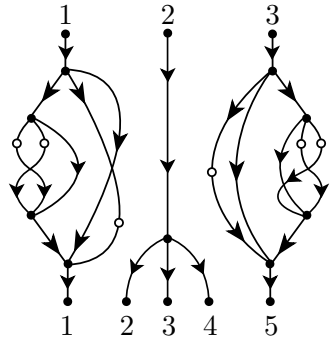

\begin{defi}[Concatenation]
\label{conc}
The \textit{concatenation} of a $(p,q,n)$-strand diagram with a $(q,r,n)$-strand diagram is a $(p,r,n)$-strand diagram which is obtained by gluing the $i$-th main sink of the former with the $i$-th main source of the latter, and removing the identified vertices.
\end{defi}

It is straightforward to check that set of $(p,q,n)$-strand diagrams is a grupoid with the operation of concatenation. In addition, we can \textit{close} a $(r,r,n)$-strand diagram by identifying every $i$-th main sink with the $i$-th main source. The result is a closed $n$-strand diagram. 

\begin{remark}
All reductions are defined in the same way for $(p,q,n)$-strand diagrams, since they only act on a small piece of a graph. Likewise, Definitions \ref{def:equiv} and \ref{def:red} can be applied also to $(p,q,n)$-strand diagrams, so Proposition \ref{prop:equiv} holds for them.
\end{remark}

If $\Gamma$ is a $(p,q,n)$-strand diagram (resp. a closed $n$-strand diagram), we denote by $[\Gamma]$ the class of all $(p,q,n)$-strand diagrams (resp. closed $n$-strand diagrams) which are equivalent to $\Gamma$. We say that two equivalence classes $[\Gamma], [\Gamma']$ are conjugate if there exist two elements $\Gamma \in [\Gamma]$ and $\Gamma' \in [\Gamma']$ which are conjugate.

\begin{thm}\label{TSD}
Let $\Gamma$ be a $(p,p,n)$-strand diagram, and let $\Gamma'$ be a $(q,q,n)$-strand diagram. Then $[\Gamma]$ and $[\Gamma']$ are conjugate if and only if $\overline{\Gamma}$ and $\overline{\Gamma'}$ are equivalent or differ by a finite number of conjugating transformations.
\end{thm}

\begin{proof}[Proof of Theorem \ref{MT}]
It follows from Theorem \ref{TSD} and the bijection between reduced $(1,1,n)$-strand diagrams and reduced tree-pair representatives. Observe that if two $(1,1,n)$-strand diagrams are conjugate in the grupoid, the conjugating element must be another $(1,1,n)$-strand diagram. Therefore, the conjugating element corresponds to another element of some Thompson-like group. \end{proof}

The following two propositions will allow us to prove Theorem \ref{TSD}.

\begin{prop}\label{AP1}
Let $\Gamma$ be a $(p,p,n)$-strand diagram, and let $\overline{\Gamma}_r$ be a closed $n$-strand diagram obtained by applying a reduction to $\overline{\Gamma}$. Then there exists a $(q,q,n)$-strand diagram $\Gamma'$ such that $\overline{\Gamma'} = \overline{\Gamma}_r$ and $[\Gamma]$ and $[\Gamma']$ are conjugate. 
\end{prop}

Observe that Proposition \ref{AP1} can be used multiple times: if we apply another reduction $ \overline{\Gamma}_r \rightarrow \overline{\Gamma}_{r^2}$, the new $\Gamma'_r$ is obtained by conjugating $\Gamma'$.

\begin{proof}
If the reduction performed on $\overline{\Gamma}$ can be also performed on $\Gamma$, then $\Gamma' = \Gamma$. If not, we have several cases, depending on the type of reduction performed and whether the reduction involves $n$ main sinks and sources or more. Since the arguments used in this proof are essentially the same in all cases, we will explain in detail only the first one. We encourage the reader to keep in mind Figures \ref{fig:contexpproof} and \ref{fig:merbridproof} as help. 

\textbf{Type I reduction.}

\textbf{I.a} It occurs when $n$ main sinks are identified with $n$ main sources in $\Gamma$, as there are two subgraphs $\Gamma_1$, $\Gamma_2 \subset \Gamma$ whose concatenation produces a type I reduction. On the one hand, $\Gamma_1$ is composed by an $\varepsilon$-neighbourhood of a subgraph between $n$ main sources and one $n$-merge, such that the $\sigma(i)$-th edge of the $n$-merge has the $i$-th or $\sigma(i)$-th main source as endpoint, or has a $\sigma$-vertex as endpoint whose $0$-th edge has the $i$-th or $\sigma(i)$-th main source as endpoint, $\forall i \in \{1, \dots, n\}$. On the other hand, $\Gamma_2$ is composed by an $n$-split and $n$ main sinks in which occurs the same as in $\Gamma_1$. We have the following symmetry conditions:
\begin{enumerate}
\item If the $\sigma(i)$-th edge of the $n$-merge has the $i$-th main source as endpoint in $\Gamma_1$, then the $i$-th edge of the $n$-split has the $i$-th main sink as endpoint in $\Gamma_2$.
\item If there is a $\sigma$-vertex between the $i$-th main source and the $n$-merge in $\Gamma_1$, then there are no $\sigma$-vertices between the $n$-split and the $i$-th main sink in $\Gamma_2$.
\end{enumerate}
In fact, (1) and (2) hold exchanging $n$-merges by $n$-splits and main sources by main sinks. We conjugate $\Gamma$ by an element $\Gamma''$ which is a copy of $\Gamma_1$, so the concatenation $\Gamma_2 \Gamma''$ produces the same type I reduction as in $\Gamma$. Besides, $(\Gamma'')^{-1}\Gamma_1$ produces an identity type I reduction.

\textbf{I.b} If there are more than $n$ main sources and $n$ main sinks glued together in the type I reduction, we can do the same as in the case I.a multiple times, one for every set of $n$ main sources, as Figure \ref{fig:contexpproof} shows.

\textbf{Type II reduction.}

\textbf{II.a} The reduction occurs when $\Gamma$ ends with an $n$-merge followed by a main sink. On the other hand $\Gamma$ starts with a main source followed by a $\sigma$-vertex followed by an $n$-split. In this case we conjugate $\Gamma$ by an element $\Gamma''$ which has a main source followed by an $n$-split.

\textbf{II.b} There are more than one main sources and main sinks involved in the reduction and one $\sigma$-vertex after one of the main sources. We conjugate $\Gamma$ by an element which glues an $n$-split to every main sink involved.

\textbf{Type III reduction.}

\textbf{III} In this case $\Gamma$ ends with a $\sigma_1$-vertex followed by a main sink, and starts with a main source followed by a $\sigma_2$-vertex. We conjugate $\Gamma$ by an element which has a main source followed by a $\sigma_2$-vertex followed by an $n$-split. 

\textbf{Type IV reduction.}

\textbf{IV.a} If it involves only $n$ main sources and $n$ main sinks and the closure of $\Gamma$ does not produce free loops, we do the same process as in the case I.a.

\textbf{IV.b} If it involves only $n$ main sources and $n$ main sinks and the closure of $\Gamma$ produces free loops, we conjugate $\Gamma$ by an element which has $n$ main sources followed by an $n$-merge.

\textbf{IV.c} If the type IV reduction involves more than $n$ main sources and $n$ main sinks we proceed as in the case I.b. \end{proof}
\begin{figure}[h]
\labellist
\pinlabel \textbf{I.a} at 35 0
\pinlabel \textbf{I.b} at 125 0
\pinlabel \textbf{II.a} at 215 0
\pinlabel \textbf{II.b} at 300 0
\pinlabel $\sigma^{-1}$ at 365 156
\pinlabel $\sigma$ at 360 127
\pinlabel $\sigma$ at 360 69
\pinlabel $\sigma$ at 360 46
\pinlabel $\Gamma$ at -10 101
\pinlabel $\Gamma_1$ at 20 115
\pinlabel $\Gamma_2$ at 20 89
\pinlabel $\Gamma_1$ at 20 27
\pinlabel $\Gamma^{-1}_1$ at 20 170

\pinlabel $\Gamma_1$ at 170 115
\pinlabel $\Gamma_2$ at 85 89
\pinlabel $\Gamma^{-1}_2$ at 87 27
\pinlabel $\Gamma_1$ at 170 27
\pinlabel $\Gamma^{-1}_1$ at 171 170
\pinlabel $\Gamma_2$ at 85 170

\pinlabel $\Gamma''$ at -10 27
\pinlabel $(\Gamma'')^{-1}$ at -15 170
\endlabellist
\centering
\includegraphics[width=0.85\textwidth]{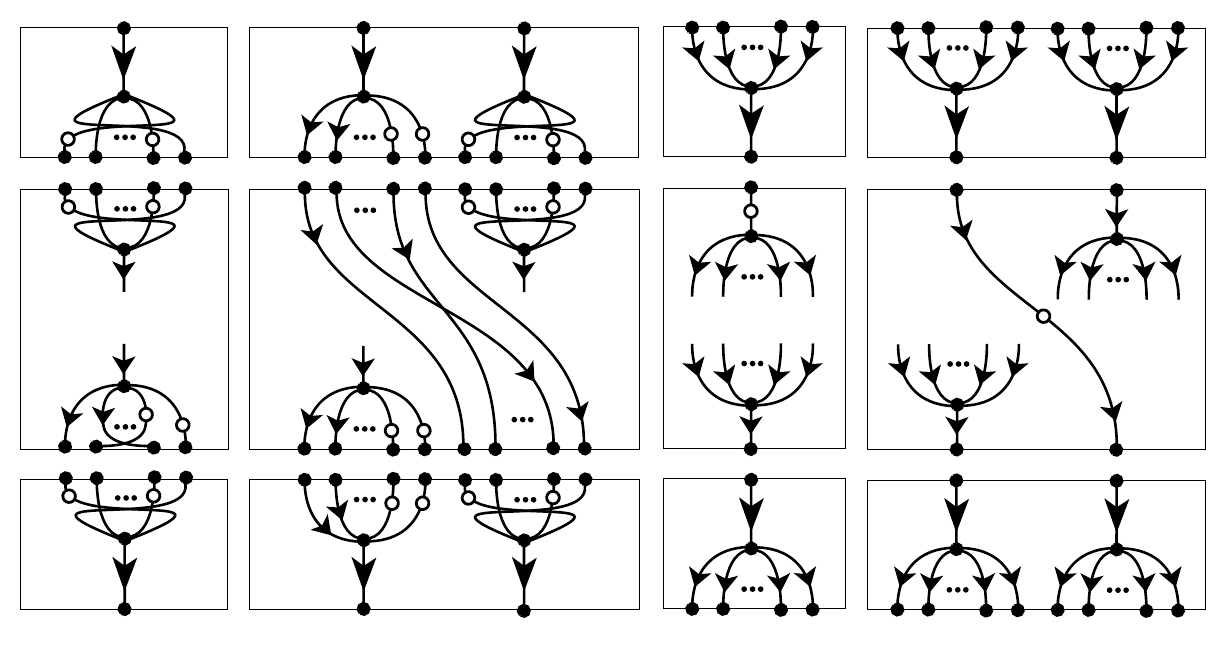}
\caption{Type I and II reductions.}
\label{fig:contexpproof}
\end{figure}
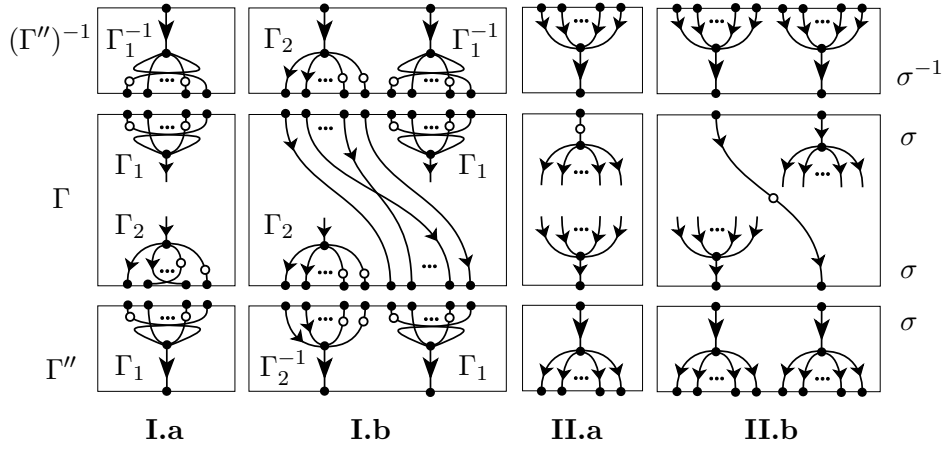

\begin{figure}[h]
\labellist
\pinlabel \textbf{III} at 33 0
\pinlabel \textbf{IV.a} at 90 0
\pinlabel \textbf{IV.b} at 148 0
\pinlabel \textbf{IV.c} at 255 0
\pinlabel $\Gamma$ at -10 97
\pinlabel $\Gamma''$ at -10 33
\pinlabel $(\Gamma'')^{-1}$ at -15 163
\pinlabel $\sigma_2^{-1}$ at 46 157
\pinlabel $\sigma_2$ at 45 127
\pinlabel $\sigma_1$ at 45 70
\pinlabel $\sigma_2$ at 45 43
\pinlabel $\sigma$ at 335 135
\pinlabel $\sigma$ at 335 72
\pinlabel $\sigma^{-1}$ at 340 156
\pinlabel $\sigma$ at 335 50
\endlabellist
\centering
\includegraphics[width=0.85\textwidth]{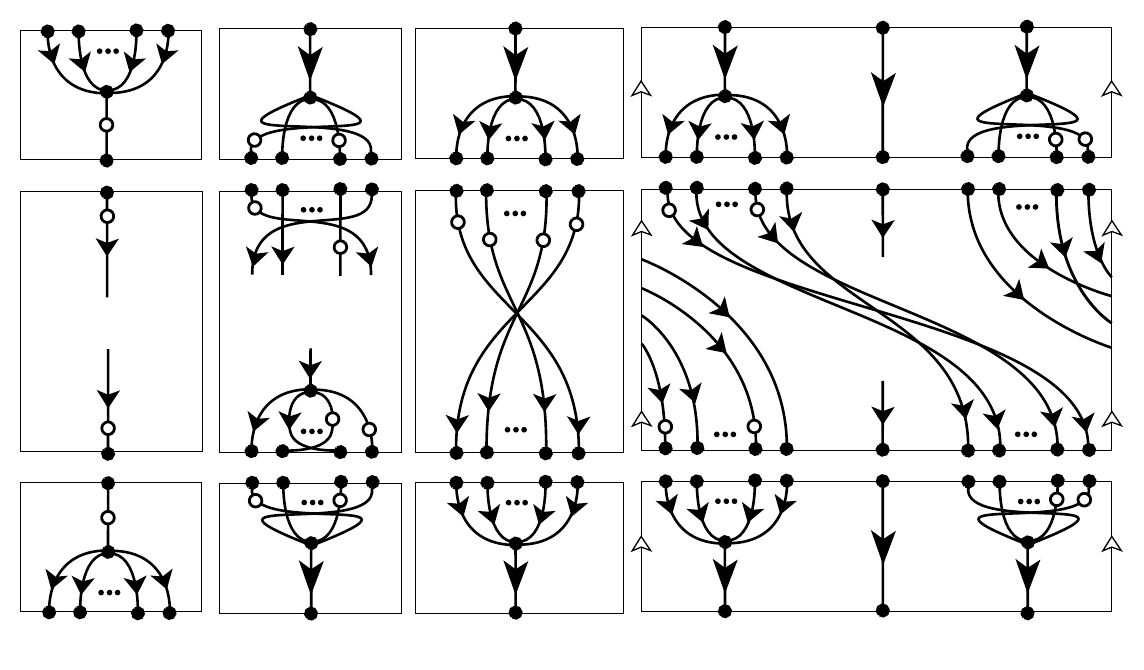}
\caption{Type III and IV reductions.}
\label{fig:merbridproof}
\end{figure} 

\begin{prop}\label{AP2}
Let $\Gamma$ be a $(p,p,n)$-strand diagram, and let $\Gamma'$ be a $(q,q,n)$-strand diagram. If $\overline{\Gamma} = \overline{\Gamma'}$, then $[\Gamma]$ and $[\Gamma']$ are conjugate.
\end{prop}

The following proof is taken word for word of \cite[Proposition 3.3]{BM}, included here for completeness:
\begin{proof}
Let ${\Gamma}^{\infty}$ be the infinite $n$-strand diagram obtained when we concatenate infinitely many times $\Gamma$ with a copy of itself. Thus ${\Gamma}^{\infty}$ is an infinite directed graph with a rotation system. In addition, ${\Gamma}^{\infty}$ is a cyclic cover of $\overline{\Gamma}$. Since $\overline{\Gamma} = \overline{\Gamma'}$, then ${\Gamma}^{\infty}$ and ${\Gamma'}^{\infty}$ must be equal. Therefore, ${\Gamma}^{\infty} = \bigcup_{k \in \mathbb{Z}} \Gamma_k = \bigcup_{k \in \mathbb{Z}} \Gamma'_k$, where $\Gamma_k = \Gamma$ and $\Gamma'_k = \Gamma'$ $\forall k $, so there exists a deck transformation $t$ such that $t(\Gamma_k)=\Gamma_{k+1}$ and $t(\Gamma'_k)=\Gamma'_{k+1}$ $\forall k$. Let:
$$ \displaystyle \Gamma''_l = \bigg( \bigcup_{k \geq 0}\Gamma_k \bigg) \cap \bigg( \bigcup_{k < N} \Gamma'_k \bigg) \ \mbox{and} \ \Gamma''_r = \bigg( \bigcup_{k > 0}\Gamma_k \bigg) \cap \bigg( \bigcup_{k \leq N} \Gamma'_k \bigg).$$
It follows that $\Gamma''_l$ and $\Gamma''_r$ are $n$-strand diagrams, with $t(\Gamma''_l) = \Gamma''_r$, so they are equal to the same $n$-strand diagram $\Gamma''$. Since $\Gamma_0 \cup \Gamma''_r = \Gamma''_l \cup \Gamma'_N$, we have $\Gamma\Gamma'' = \Gamma''\Gamma'$.
\end{proof}
Now, we are ready to prove Theorem \ref{TSD}:

\begin{proof}[Proof of Theorem \ref{TSD}]
On the one hand, if $\Gamma$ and $\Gamma'$ are conjugate, then $\Gamma = (\Gamma'')^{-1}\Gamma'\Gamma''$ for some $(p,q,n)$-strand diagram $\Gamma''$. Then the closures of both $\Gamma = (\Gamma'')^{-1}\Gamma'\Gamma''$ and $\Gamma'$ are equal or they differ by some conjugating transformations, depending on the order in which the type III reductions are performed when $(\Gamma'')^{-1}$ and $\Gamma''$ are glued together in $\overline{\Gamma}$.

On the other hand, suppose that the closures of $\Gamma$ and $\Gamma'$ are equivalent or they differ by some conjugating transformations. By Remark, \ref{rem:conjugatingtransformation}, we can perform conjugating transformations without changing the conjugacy class of both $\Gamma$ and $\Gamma'$. So, suppose that $\Gamma$ and $\Gamma'$ are equivalent. It means that there exists an element $\Gamma''$ which is the reduction of both. By Proposition \ref{AP1}, there exist $\Gamma_1$ and $\Gamma_2$ such that $\Gamma'' =\overline{\Gamma_1} = \overline{\Gamma_2}$ and $[\Gamma_1]$, $[\Gamma_2]$ are conjugate to $[\Gamma]$ and $[\Gamma']$ respectively. Finally by Proposition \ref{AP2}, $[\Gamma_1]$ and $[\Gamma_2]$ are conjugate, so $[\Gamma]$ and $[\Gamma']$ are conjugate too.
\end{proof}


\section{A non-isomorphism result}

In this section, we give an application of Theorem \ref{MT}, which generalizes a result of Higman \cite{HI}. We start by proving the following:

\begin{prop}\label{redclos}
Let $v \in V_n(P)$ be an element of prime order such that $\ord(v)$ does not divide $\ord(P)$. Let $\Gamma$ be its associated $n$-strand diagram. Then the reduced closure of $\Gamma$ has no $\sigma$-vertices.
\end{prop}

In particular, this result says that such an element $v \in V_n(P)$ is conjugate to an element of $V_{n}(Id) = V_n$.

\begin{proof}
Let $v \sim (\mathcal{T}_\mathcal{B},\mathcal{T}_{v(\mathcal{B})}, \tau, \sigma)$ be reduced. We consider the sets of leaves $\{v^k(B)\}_{0 \leq k \leq p}$ $\forall B \in \mathcal{B}$, that is, the set of orbits of $v$ inside $\mathcal{B}$. Since $v$ is an element of prime order $p$, every orbit contains exactly one leaf or $p$ different leaves. In the first case, the orbit corresponds in $\Gamma$ to a free loop which contains a number $k \leq p$ of $\sigma_i$-vertices such that $\sigma_k \circ \dots \circ \sigma_1 = Id$. By performing type III and IV reductions we obtain a free loop with no $\sigma$-vertices. In the second case $v$ acts as the identity on the leaf, since $\ord(\sigma) \vert \ord(P)$ $\forall \sigma \in P$. Therefore, the closure of $[\Gamma]$ consists only of free loops without $\sigma$-vertices. 
\end{proof}

In order to prove Theorem \ref{thm:nonisomorphism} we need to study the conjugacy classes of isomorphisms of $H$ into $V_n(P)$ where $H$ is a finite group, as done in \cite{HI}.

\begin{defi}[$H$-space]
Given a group $H$, an \textit{$H$-space} is a set on which $H$ acts. We say that the $H$-space is \textit{transitive} if the action is transitive.
\end{defi}

\begin{lemma}\cite[Lemma 6.1]{HI}\label{congruence} Let $S_1,\dots,S_t$  be a set of representatives of the isomorphism classes of transitive $H$-spaces. Then the conjugacy classes of $H$ into $V_{n}(Id)$ are in bijective correspondence with the solutions of $$n_1\vert S_1\vert + \dots + n_t \vert S_t \vert \equiv^* 1 \mod n-1,$$ where $a \equiv^* b \mod n-1$ means that $a \equiv b \mod n-1$ and $a = 0 \iff b = 0$.
\end{lemma}

We extend a result \cite[Lemma 6.2]{HI} from $V_n(Id)$ to $V_n(P)$ as follows:

\begin{lemma}\label{lem:orderpelements}
If $p$ is a prime which divides neither $n-1$ nor $\ord(P)$, then the number of conjugacy classes of elements of order $p$ in $V_n(P)$ is $n$.
\end{lemma}
\begin{proof}
The key idea of the proof is that the congruence of Lemma \ref{congruence} holds for $V_n(P)$ in this case. By Lemma \ref{redclos}, the tree-pair representative of the elements of $V_n(P)$ which represent the generators of these conjugacy classes of homeomorphisms are always conjugate to an element of $V_n(Id)$. So the number of conjugacy classes of homeomorphisms of a cyclic group of order $p$ in $V_n(P)$ is equal to the number of solutions of $$n_1 + pn_2 \equiv^* 1 \mod n-1.$$
For every $n_1 \in \{0,2,3,\dots, n-1\}$ there is only one possible solution for $n_2$, since $p$ does not divide $n-1$. For $n_1 \equiv^* 1$, where we have $n_2 \equiv^* 0$ and $n_2 \equiv^* n-1$. Apart from the trivial homomorphism, there are $n$ different conjugacy classes of elements of order $p$ in $V_n(P)$. \end{proof}

\begin{proof}[Proof of Theorem \ref{thm:nonisomorphism}]
We take $p$ to be sufficiently large so it does not divide any of $\ord(P)$, $\ord(Q)$, $n-1$ and $m-1$. By Lemma \ref{lem:orderpelements}, the number of conjugacy classes of elements of order $p$ in $V_n(P)$ and $V_m(Q)$ are $n$ and $m$ respectively, so these two groups cannot be isomorphic.
\end{proof}

\section*{Acknowledgements}
The author would like to thank his advisor Javier Aramayona for conversations and support. He is also grateful to Motoko Kato and Diego López for comments; and Waltraud Lederle for helping to strengthen Theorem \ref{thm:nonisomorphism}. Finally, he also wants to acknowledge financial support from the Spanish Ministry of Economy and Competitiveness, through the “Severo Ochoa Programme for Centres of Excellence in R\&D” (SEV-2015-0554) and the grant MTM2015-67781.

\label{Bibliography}

\bibliographystyle{abbrv} 
\bibliography{Bibliography}

\end{document}